\documentclass[12pt]{amsart}
\usepackage{amsmath,amssymb,amsfonts,latexsym}
\usepackage{eucal}
\usepackage{amsthm}
\usepackage{mathrsfs}

\newtheorem{theorem}{Theorem}[section]
\newtheorem{proposition}[theorem]{Proposition}
\newtheorem{lemma}[theorem]{Lemma}

\newtheorem{corollary}[theorem]{Corollary}

\input amssym.def
 
\input amssym.tex

\renewcommand{\SS}{\mathbb{S}}
\newcommand{\RR}{\mathbb{R}}

\newcommand{\KK}{\mathscr{K}}

\begin{document}

\title{Steiner symmetrization using a finite set of directions}
\author{Daniel A. Klain}
\address{Department of Mathematical Sciences,
University of Massachusetts Lowell,
Lowell, MA 01854 USA}
\email{Daniel\_{}Klain@uml.edu}


\begin{abstract} Let $v_1, \ldots, v_m$ be a finite set of unit vectors in $\RR^n$.
Suppose that an infinite sequence of Steiner symmetrizations
are applied to a compact convex set $K$ in $\RR^n$, where each of the symmetrizations is taken
with respect to a direction from among the $v_i$.
Then the resulting sequence of Steiner symmetrals always converges, and the limiting body is 
symmetric under reflection in any of the directions $v_i$ that appear infinitely often in the sequence.   
In particular, an infinite
periodic sequence of Steiner symmetrizations always converges, and the set functional
determined by this infinite process is always idempotent.
\end{abstract}


\maketitle

\section{Introduction}

Denote $n$-dimensional Euclidean space by $\RR^n$, and 
let $\KK_n$ denote the set of all compact convex sets in $\RR^n$.
Let $K \in \KK_n$, and let $u$ be a unit vector.
Viewing $K$ as a family of line segments parallel to $u$, slide these segments along $u$ so that
each is symmetrically balanced around the hyperplane $u^\perp$.  By Cavalieri's principle, the volume
of $K$ is unchanged by this rearrangement.  
The new set, called the {\em Steiner symmetrization} of $K$ in the direction of $u$, 
will be denoted by ${\rm s}_u K $.  It is not difficult to show that 
${\rm s}_u K $ is also convex, and that ${\rm s}_u K  \subseteq {\rm s}_u L $ whenever $K \subseteq L$.  
A little more work verifies
the following intuitive assertion: if you iterate Steiner symmetrization of $K$ through
a suitable sequence of unit directions, the successive Steiner symmetrals of $K$ will
approach a Euclidean ball in the Hausdorff topology on compact (convex) subsets of $\RR^n$.
A detailed proof of this assertion can be found in any of 
\cite[p. 98]{Egg}, \cite[p. 172]{Gru-book}, or \cite[p. 313]{Webster}, for example.  

For well over a century Steiner symmetrization has played a fundamental role in answering
questions about isoperimetry and related geometric inequalities \cite{Gard-BM,Gard2006,Steiner,Talenti}.  
Steiner symmetrization appears explicitly in
the titles of numerous papers (see e.g.
\cite{bia-gro,BKLYZ,BLM,burch,CCF,CF1,CF2,Falconer,GardX,KM1,KM2,Long,Mani,McNabb,Scott,volcic-symm})
and plays a key role in recent work such as \cite{burch-fort,Hab-Sch,LYZ-Orl-proj,vansch1,vansch2}.

In spite of the importance of Steiner symmetrization throughout geometric analysis, 
many elementary questions about this construction remain open, including some 
concerning the following issue: 
Given a convex body $K$, under what conditions on the sequence
of directions $u_i$ does the sequence of Steiner symmetrals 
\begin{align}
{\rm s}_{u_i} \cdots {\rm s}_{u_1} K
\label{stseq}
\end{align}
converge?  And if the sequence converges, what symmetries are satisfied by the limiting body?

The sequence of bodies~(\ref{stseq}) is called a {\em Steiner process}.  If the limit
\begin{align}
\lim_{i\rightarrow\infty} {\rm s}_{u_i} \cdots {\rm s}_{u_1} K
\end{align}
exists, the resulting body $\tilde{K}$ is called the limit of that Steiner process.
In \cite{BKLYZ} it is shown that not every Steiner process converges, even if the directions $u_i$ are
dense in the sphere.

This article addresses the case in which an infinite Steiner process of the form~(\ref{stseq}) uses only
a finite set of directions, each repeated infinitely often, whether in a periodic fashion, according to some
more complex arrangement, or even completely at random.

Let $v_1, \ldots, v_m$ be a finite set of unit vectors in $\RR^n$.
Suppose that an infinite sequence of Steiner symmetrizations
is applied to a compact convex set $K$ in $\RR^n$, where each of the symmetrizations is taken
with respect to a direction from among the $v_i$.
The main result of this article is Theorem~\ref{findir}, which asserts that 
the resulting sequence of Steiner symmetrals {\em always} converges.  The limiting body is 
symmetric under reflection in any of the directions $v_i$ that appear infinitely often in the sequence.   
In particular, an infinite
periodic sequence of Steiner symmetrizations always converges, and the set functional
determined by this infinite process is always idempotent.

\section{Background and basic properties of Steiner symmetrization}

Given a compact convex set $K$ and a unit vector $u$, we have 
 ${\rm s}_u K  = K$ (or respectively, up to
translation)
if and only if $K$ is symmetric under reflection across the subspace $u^\perp$ (respectively, up to
translation).  In particular, ${\rm s}_u K = K$ will
hold for {\em every} direction $u$ (or even a dense set of directions) if and only if $K$ is a Euclidean ball
centered at the origin.

Let $h_K: \RR^n \rightarrow \RR$ denote the support function of a compact convex set $K$;
that is,
$$h_K(v) = \max_{x \in K} x \cdot v.$$
The standard separation theorems of convex geometry imply that
the support function $h_K$ characterizes the body $K$; that is, $h_K = h_L$ if and only if $K=L$.
If $K_i$ is a sequence in $\KK_n$, then $K_i \rightarrow K$ in the Hausdorff topology if and only if
$h_{K_i} \rightarrow h_K$ uniformly when restricted to the unit sphere in $\RR^n$.

Given compact convex subsets $K,L \subseteq \RR^n$
and $a,b \geq 0$, denote
$$aK + bL = \{ax + by \; | \; x \in K \hbox{ and } y \in L\}.$$
An expression of this form is called a {\em Minkowski combination} or 
{\em Minkowski sum}.  Since $K$ and $L$ are convex sets, the set $aK + bL$ is also convex.  
Convexity also implies that $aK + bK = (a+b)K$ for all $a,b \geq 0$, although this does not hold for general
sets.  Support functions satisfy the identity
$h_{aK+bL} = ah_K + bh_L$.  (See, for example, any of 
\cite{Bonn2,red,Webster}).

The following is also easy to prove (see, for example, \cite[p. 169]{Gru-book} or \cite[p. 310]{Webster}).
\begin{proposition}
$${\rm s}_u(K+L) \supseteq {\rm s}_u K + {\rm s}_u L.$$
\label{steinsum}
\end{proposition}

Denote by $V_n(K)$ the $n$-dimensional volume of a set $K \subseteq \RR^n$. 
Given $K, L \in \KK_n$ and $\varepsilon > 0$, the function $V_n(K + \varepsilon L)$ is a polynomial in $\varepsilon$, 
whose coefficients are given by {\em Steiner's formula} \cite{Bonn2,red,Webster}.  
In particular, the following derivative is well defined:
\begin{align}
nV_{n-1,1}(K,L) & \; = \; \lim_{\varepsilon \rightarrow 0} \frac{V_n(K + \varepsilon L) - V_n(K)}{\varepsilon} 
\; = \; \left. \frac{d}{d \varepsilon} \right|_{\varepsilon = 0} V_n(K + \varepsilon L).
\label{dv}
\end{align}
The expression $V_{n-1,1}(K,L)$ is an example of a {\em mixed volume} of $K$ and $L$.
Important special cases appear when either of $K$ or $L$ is a unit Euclidean ball $B$:
\begin{align}
\label{surf}
nV_{n-1,1}(K,B) & = \hbox{Surface Area of } K \\ \notag
\tfrac{2}{\omega_n}V_{n-1,1}(B,L) & = \hbox{Mean Width of } L 
\end{align}
where $\omega_n$ denotes the $n$-volume of the Euclidean unit ball $B$.  We will denote the mean width of $L$ by $W(L)$.

It follows from Proposition~\ref{steinsum} and the volume invariance of Steiner symmetrization that
$$V_n(K+\varepsilon L) = V_n\Big( {\rm s}_u(K+\varepsilon L) \Big) \geq V_n( {\rm s}_u K+\varepsilon {\rm s}_u L ),$$
so that
$$\frac{V_n(K+\varepsilon L) - V_n(K)}{\varepsilon} 
\geq \frac{V_n \Big( {\rm s}_u K +\varepsilon {\rm s}_u L \Big) - V_n({\rm s}_u K)}{\varepsilon},$$
for all $\varepsilon > 0$.  Letting $\varepsilon \rightarrow 0^+$, we have
\begin{align}
V_{n-1,1}(K,L) \geq V_{n-1,1}({\rm s}_u K, {\rm s}_u L)
\label{reduce}
\end{align}
for all $K, L \in \KK_n$ and all unit directions $u$.  

For $r \geq 0$ denote by
$rB$ the closed Euclidean ball of radius $r$ centered at the origin.
Since ${\rm s}_u B = B$, it follows from~(\ref{surf}) and~(\ref{reduce}) that 
the surface area does not increase under Steiner symmetrization.  Similarly, 
the mean width satisfies $W({\rm s}_u K) \leq W(K)$ for all $u$.

From monotonicity it is also clear that, if $r,R \in \RR$ such that 
\begin{align}
rB \subseteq K \subseteq RB
\label{betwrR}
\end{align}
then 
\begin{align}
rB \subseteq {\rm s}_u K \subseteq RB.
\label{betst}
\end{align}
Let $R_K$ denote the minimum radius of any Euclidean $n$-ball containing $K$, and let $r_K$ denote the maximal radius
of any Euclidean $n$-ball contained inside $K$.  It follows that
\begin{align}
R_{{\rm s}_u K} \leq R_K \quad \hbox{ and } \quad  r_K \leq r_{{\rm s}_u K}
\label{rRsu}
\end{align}
It can also be shown using elementary arguments that Steiner symmetrization does not increase
the diameter of a set \cite[p. 310]{Webster}. 

The following lemma will be useful in Section~\ref{finsetdir}.
\begin{lemma} Suppose that $\{K_i\}$ is a convergent sequence of compact convex sets whose limit $K$
has nonempty interior.  Then, for all $\,0 < \tau < 1$, there is an integer $N > 0$ such that 
$$(1-\tau)K \subseteq K_i \subseteq (1+\tau)K$$
for all $i \geq N$.  
\label{pmK}
\end{lemma}

\begin{proof} Since $K$ has interior, 
it has positive inradius $r$.  Without loss of generality (translating as needed) we may assume that $rB \subseteq K$.
For $\tau \in (0,1)$, choose $N$ so that 
$$K_i \subseteq K + r\tau B \quad \hbox{ and } \quad K \subseteq K_i + r\tau B$$
for $i \geq N$.  In this case,
$$K_i \subseteq K + r\tau B \subseteq K + \tau K = (1 + \tau) K$$
and 
$$K \subseteq K_i + r\tau B \subseteq K_i + \tau K,$$
so that  $(1 - \tau) K \subseteq K_i$.
\end{proof}
It follows from Lemma~\ref{pmK} and the monotonicity property~(\ref{betst}) that 
Steiner symmetrization is continuous with respect to $K$ and $u$ provided 
that $K \in \KK_n$ has nonempty interior.  (See also \cite[p. 171]{Gru-book} or \cite[p. 312]{Webster}.)  

Note that
the interior condition is needed to guarantee continuity: 
Steiner symmetrization is {\em not} continuous at lower-dimensional sets.
For example, consider a sequence of distinct unit line segments $K_i$ with endpoints at $\pm u_i$, where $u_i \rightarrow u$.  While the line segments $K_i$ approach the line segment with endpoints at $\pm u$, 
their symmetrizations
${\rm s}_u K_i$ form a sequence of projected line segments in $u^\perp$ whose lengths approach zero, so that
${\rm s}_u K_i \rightarrow o$, the origin.
But ${\rm s}_u K = K \neq o$, since $K$ is already symmetric under reflection across $u^\perp$.
See also \cite[p. 170]{Gru-book}.

Denote by $\KK^n_{r,R}$ the set of compact convex sets in $\RR^n$ satisfying~(\ref{betwrR}).
By the Blaschke selection theorem $\KK^n_{r,R}$ is compact.  Since $\SS^n$ is also compact, the function
$$(K, u) \mapsto {\rm s}_u K$$
is uniformly continuous on $\KK^n_{r,R} \times \SS^{n-1}$.

Moreover, it follows from monotonicity that 
Steiner symmetrization does respect the limits of decreasing sequences of sets, 
even if the limit has empty interior.  More specifically, recall that if
\begin{align}
K_1 \supseteq K_2 \supseteq K_3 \supseteq \cdots
\label{desc}
\end{align}
then
\begin{align}
\lim_{m \rightarrow \infty} K_m = \bigcap_{m=1}^\infty K_m.
\label{decr-inter}
\end{align}
This follows from the fact that a pointwise limit of support functions of 
compact convex sets is always
a uniform limit as well \cite[p. 54]{red}.  
We then have the following special case where continuity holds for Steiner symmetrization
of a {\em descending} sequence of convex bodies, even when the limiting body is lower dimensional.
\begin{proposition} Suppose that $\{K_m\}$ is a sequence of compact convex sets in $\RR^n$ such 
that~(\ref{desc}) holds, and let 
$$K = \lim_{m \rightarrow \infty} K_m = \bigcap_{m=1}^\infty K_m.$$
If $u$ is a unit vector in $\RR^n$, then
$${\rm s}_u K = \lim_{m \rightarrow \infty} {\rm s}_u K_m = \bigcap_{m=1}^\infty {\rm s}_u K_m.$$
\label{descend}
\end{proposition}

\begin{proof} 
Denote by $\pi_u L$ the orthogonal projection of a compact convex set $L$ onto the subspace $u^\perp$, and note that
$\pi_u {\rm s}_u L = \pi_u L$  for all $L \in \KK_n$.  It follows from the monotonicity of ${\rm s}_u$ applied to the 
sequence~(\ref{desc}) that
$${\rm s}_u K_1 \supseteq {\rm s}_u K_2 \supseteq {\rm s}_u K_3 \supseteq \cdots,$$
so that the limit 
$$L = \lim_{m \rightarrow \infty} {\rm s}_u K_m = \bigcap_{m=1}^\infty {\rm s}_u K_m,$$
exists.  Moreover, since $K \subseteq K_m$ for all $m$, it follows that
${\rm s}_u K \subseteq {\rm s}_u K_m$ as well, so that ${\rm s}_u K \subseteq L$.  Note also
that both ${\rm s}_u K$ and $L$ are symmetric under reflection across $u^\perp$.

From the continuity of orthogonal projection we also have
$$\pi_u {\rm s}_u K = \pi_u K = \lim_{m \rightarrow \infty} \pi_u K_m
= \lim_{m \rightarrow \infty} \pi_u {\rm s}_u K_m = 
\pi_u \lim_{m \rightarrow \infty} {\rm s}_u K_m = \pi_u L,$$
so that ${\rm s}_u K$ and $L$ have the same orthogonal projection into $u^\perp$.

Finally, for each $x \in \pi_u L$, the linear slice of $L$ perpendicular to $x$ 
has length given by the infimum over $m$ of the length of the linear slice of ${\rm s}_u K_m$ over the point $x$.
Since Steiner symmetrization translates these slices (preserving their lengths), this is the same
as the infimum over $m$ of the length of the linear slice of $K_m$ over the point $x$, which gives
the length of linear slice of ${\rm s}_u K$ perpendicular to $x$.  Hence, $L = {\rm s}_u K$. 
\end{proof}

\section{The layering function}

Define the {\em layering function} of $K \in \KK_n$ by 
$$\Omega(K) = \int_{0}^\infty V_n(K \cap rB) \, e^{-r^2} \; dr$$
Evidently the function $\Omega$ is monotonic and continuous on $\KK_n$.
The layering function 
vanishes on sets with empty interior 
and is strictly positive on sets with non-empty interior.

The following crucial property of Steiner symmetrization will be used in the sections that follow.
\begin{theorem}  Suppose that $K \in \KK_n$, and   
let $u$ be a unit vector.  Then
\begin{align}
\Omega({\rm s}_u K) \geq \Omega(K).
\label{layer}
\end{align}
If $K$ has non-empty interior, then equality holds in~(\ref{layer}) if and only if ${\rm s}_u K = K$.
\label{dend}
\end{theorem}

In the proof of Theorem~\ref{dend} we will use the following elementary fact:  If $D$ is a ball centered at the origin,
and if $X$ is a line segment, parallel to the unit vector $u$, having one endpoint in the interior of $D$ and the other endpoint outside $D$, then Steiner symmetrization will strictly increase the slice length; that is,
\begin{align}
|{\rm s}_u X \cap D| > |X \cap D|.
\label{slide}
\end{align}
To see this, let $\ell$ denote the line through $X$.  Our conditions on the endpoints of $X$ 
imply that $|\ell \cap D| > |X \cap D|$.
Meanwhile, ${\rm s}_u$ fixes $D$ and slides $X$ parallel to $u$ until it is symmetric about $u^\perp$.
If $|X| < |\ell \cap D|$, then ${\rm s}_u X$ will lie wholly inside $D$,
so that $|{\rm s}_u X \cap D| = |X| > |X \cap D|$ and (\ref{slide}) follows.
If $|X| \geq |\ell \cap D|$, then ${\rm s}_u X$ will cover the slice $\ell \cap D$ completely,
so that $|{\rm s}_u X \cap D| = |\ell \cap D|$
and~(\ref{slide}) follows once again.

\begin{proof} Let $u$ be a unit vector.  The monotonicity of ${\rm s}_u$ implies that
$${\rm s}_u (K \cap rB) \subseteq {\rm s}_u K \cap {\rm s}_u rB = {\rm s}_u K \cap rB,$$
so that 
$$V_n({\rm s}_u K \cap rB) \geq V_n({\rm s}_u (K \cap rB)) =  V_n(K \cap rB),$$
whence $\Omega({\rm s}_u K) \geq \Omega(K)$.  

Evidently equality holds if ${\rm s}_u K = K$.  For the converse, 
suppose that $K$ has non-empty interior, and that ${\rm s}_u K \neq K$.  
Let $\psi$ denote the reflection of $\RR^n$ across the subspace $u^\perp$.
Since $\psi K \neq K$ and $K$ has non-empty interior,
there is a point $x \in int(K)$ such that $\psi x \notin K$.  Let $D$ denote the ball around 
the origin of radius $|x|$, and let $\ell$ denote the line through $x$ and parallel to $u$.  
The slice $K \cap \ell$ meets the boundary of $D$ at $x$ on one side of $u^\perp$, 
has an endpoint $x + \varepsilon u$
outside $D$ and another endpoint $x - \delta u$ in the interior of $D$, where $\varepsilon, \delta > 0$.  
It follows from~(\ref{slide}) that 
$$|{\rm s}_u K \cap \ell \cap D| > |K \cap \ell \cap D|.$$  
Moreover, this holds for parallel slices through points $x'$ in an open neighborhood of $x$.
After integration of parallel slice lengths to compute volumes, we obtain 
$$V_n({\rm s}_u K \cap rB) > V_n(K \cap rB)$$
for values of $r$ in an open neighborhood of $|x|$.  It follows that $\Omega({\rm s}_u K) > \Omega(K)$.
\end{proof}

In \cite[p. 90]{Egg} Eggleston proves a result similar to Theorem~\ref{dend} for the surface area function.
If $S(K)$ denotes the surface area of a compact convex set $K$ having nonempty interior, 
then $S({\rm s}_u K) \leq S(K)$, with equality if and only if
$K$ and ${\rm s}_u K$ are translates.  The layering function $\Omega$ is more appropriate 
for our purposes, because the equality case
in Theorem~\ref{dend} is more stringent (even translates are not allowed).

\section{Steiner processes}

Let $\alpha = \{u_1, u_2, \ldots \}$ be a sequence of unit vectors in $\RR^n$.  Given $K \in \KK_n$,
denote
\begin{align}
K_i = {\rm s}_{u_i} \cdots {\rm s}_{u_i} K
\label{stseq2}
\end{align}
for $i = 1, 2, \ldots$.

\begin{proposition} 
The sequence of bodies~(\ref{stseq2}) is uniformly bounded and therefore always has a convergent subsequence.
\label{subsq}
\end{proposition}

\begin{proof}
Since $K$ is compact, there exists $\rho \geq 0$ such that $K \subseteq \rho B$.
Since Steiner symmetrization is monotonic, we have 
$${\rm s}_{u_i} \cdots {\rm s}_{u_1} K \subseteq {\rm s}_{u_i} \cdots {\rm s}_{u_1} \rho B = \rho B$$
as well, so that sequence is bounded.  The Blaschke selection theorem \cite{Bonn2,red,Webster} then
implies that~(\ref{stseq2}) has a convergent subsequence. 
\end{proof}

Note that the original sequence $\{K_i\}$ defined by~(\ref{stseq2}) does not necessarily converge
to a limit.  If $L = \lim_{i} K_i$ exists, we write $L = {\rm s}_\alpha K$.  If $L$ is the limit of
some convergent subsequence of $\{K_i\}$, we say that $L$ is a {\em subsequential limit} of ${\rm s}_\alpha K$.

Since the layering function $\Omega$ is weakly 
increasing under Steiner symmetrization by Theorem~\ref{dend} 
and is also continuous and bounded above, 
the following is immediate.
\begin{proposition} If $L$ is a subsequential limit of ${\rm s}_\alpha K$, then
$$\Omega(L) = \sup_i \Omega(K_i)$$
\end{proposition}

\begin{proposition} If ${\rm s}_\alpha M$ exists, and
if $L$ is a subsequential limit of ${\rm s}_\alpha K $, then
$$V_{n-1,1}( L, {\rm s}_\alpha M) = \inf_i V_{n-1,1}( K_i, {\rm s}_\alpha M)$$
\end{proposition}

\begin{proof} We are given that $L = \lim_j K_{i_j}$ for some subsequence $\{K_{i_j}\}$ of~(\ref{stseq2}). 
The continuity of mixed volumes implies that the sequence
\begin{align}
V_{n-1,1}( K_{i_j}, {\rm s}_{u_{i_j}} \cdots {\rm s}_{u_1} M)
\label{mixsubseq}
\end{align}
converges to $V_{n-1,1}(L, {\rm s}_\alpha M)$.
Since $V_{n-1,1}( K_i, {\rm s}_{u_i} \cdots {\rm s}_{u_1} M )$ is decreasing with respect
to $i$ by~(\ref{reduce}), the corresponding subsequence~(\ref{mixsubseq}) is also decreasing, 
and the proposition follows.
\end{proof}

In particular, we have the following.
\begin{proposition} Suppose that ${\rm s}_\alpha M $ exists.
If ${\rm s}_\alpha K $ has a subsequential limits $L_1$ and $L_2$, then 
$$V_{n-1,1}( L_1, {\rm s}_\alpha M ) = V_{n-1,1}( L_2, {\rm s}_\alpha M )$$
\label{uniqmix}
\end{proposition}

Because Steiner symmetrization may be discontinuous on sequences of bodies converging to lower dimensional
limits, the next proposition is sometimes helpful.
\begin{proposition} 
Suppose that 
$$C_1 \supseteq C_2 \supseteq C_3 \supseteq \cdots$$
is a descending sequence of compact convex sets in $\RR^n$, and denote
$$C = \bigcap_m C_m.$$
If ${\rm s}_\alpha C_m$ converges for each $C_m$, then ${\rm s}_\alpha C$ converges
to the limit
$${\rm s}_\alpha C = \bigcap_m {\rm s}_\alpha C_m.$$
\label{bigcapconv}
\end{proposition}

\begin{proof} 
Let $L$ be a subsequential limit of ${\rm s}_\alpha C$.  For each $m$ let $D_m = {\rm s}_\alpha C_m$.
Since $C \subseteq C_m$ for each $m$, the subsequential limit 
$L$ of ${\rm s}_\alpha C$ lies inside each $D_m$, so that
$$L \subseteq \bigcap_m D_m = D.$$
Meanwhile, since Steiner symmetrization does not increase mean width,
the non-negative sequence of values $W({\rm s}_{u_j} \ldots {\rm s}_{u_2} {\rm s}_{u_1} C)$ is decreasing, so that
$$\lim_j W({\rm s}_{u_j} \ldots {\rm s}_{u_2} {\rm s}_{u_1} C) 
= \inf_j W({\rm s}_{u_j} \ldots {\rm s}_{u_2} {\rm s}_{u_1} C) = \mu$$
exists.  Since $W$ is continuous, we must have $W(L) = \mu$.  It also follows from~(\ref{decr-inter}) that 
\begin{align*}
W(D) = \; \inf_m W(D_m) 
& = \; \inf_m \inf_j W \big({\rm s}_{u_j} \ldots {\rm s}_{u_2} {\rm s}_{u_1} C_m \big) \\
& = \; \inf_j \inf_m W \big({\rm s}_{u_j} \ldots {\rm s}_{u_2} {\rm s}_{u_1} C_m \big). 
\end{align*}
By Proposition~\ref{descend}, 
$${\rm s}_{u_j} \ldots {\rm s}_{u_2} {\rm s}_{u_1} C_m 
\rightarrow {\rm s}_{u_j} \ldots {\rm s}_{u_2} {\rm s}_{u_1} C,$$
so that 
$$W({\rm s}_{u_j} \ldots {\rm s}_{u_2} {\rm s}_{u_1} C_m) 
\rightarrow W({\rm s}_{u_j} \ldots {\rm s}_{u_2} {\rm s}_{u_1} C).$$
Hence,
\begin{align*}
W(D) & = \; \inf_j W({\rm s}_{u_j} \ldots {\rm s}_{u_2} {\rm s}_{u_1} C) \; = \; \mu.
\end{align*}
Since $L \subseteq D$ and $W(L) = W(D) = \mu$, it follows that $L=D$.  

We have shown that every subsequential limit of ${\rm s}_\alpha C$ has the same limit $D$.  
If the full sequence ${\rm s}_\alpha C$ does not converge, there is a subsequence $\gamma$ 
of ${\rm s}_\alpha C$ that stays some distance $\varepsilon > 0$ from $D$.  Since the
sequence ${\rm s}_\alpha C$ is uniformly bounded, so is the subsequence $\gamma$.  The Blaschke
selection theorem \cite[p. 97]{Webster} implies that $\gamma$, and therefore ${\rm s}_\alpha C$,
has a convergent subsequence $\gamma'$.  By the previous argument $\gamma'$ has limit $D$,
contradicting the construction of $\gamma$.
It follows that the original sequence ${\rm s}_\alpha C$ 
converges, and therefore must converge to the limit $D$.
\end{proof}

These results together lead to the following uniqueness theorem.
\begin{theorem} Suppose that $K \in \KK_n$ has non-empty interior.
If ${\rm s}_\alpha L = L$ for all subsequential limits $L$ of 
${\rm s}_\alpha K$ then
${\rm s}_\alpha K$ converges.
\label{fix}
\end{theorem}

\begin{proof}  By the Blaschke selection theorem,
every subsequence of ${\rm s}_\alpha K$ 
has a  sub-subsequence converging to
a limit.  Suppose that $L_1$ and $L_2$ are two such limits.

We are given that ${\rm s}_\alpha L_j  = L_j$ for each $j$.  By Proposition~\ref{uniqmix}
and the volume invariance of Steiner symmetrization, 
$$V_{n-1,1}( L_1, L_2) = V_{n-1,1}( L_2, L_2) = V_n(L_2) = V_n(K) = V_n(L_1).$$
Since $V_n(K) > 0$, the same is true of all symmetrals of $K$. 
It follows from the equality conditions of the Minkowski inequality for mixed volumes 
(see, for example, \cite{red,Webster})
that $L_1$
and $L_2$ are translates, so that $L_2 = L_1 + x$ for some $x \in \RR^n$.

Since ${\rm s}_\alpha L_j = L_j$ for each $j$, it follows that ${\rm s}_\alpha x = x$, so that
$x \in u_i^\perp$ for each $u_i \in \alpha$.  If the sequence $\alpha$ contains
a basis for $\RR^n$, then $x = 0$, and $L_1 = L_2$.  

If the sequence $\alpha$ spans a 
proper subspace $\xi$ of $\RR^n$, then $x \in \xi^\perp$.  Since every symmetrizing direction
$u_i$ of $\alpha$ lies in $\xi$, the supporting plane of $K$ normal to $x$ also supports
each symmetral $K_i$, so that 
$h_{K_i}(x) = h_K(x)$ for all $i$.  After taking limits it follows that
$$h_{L_1}(x) = h_K(x) = h_{L_2}(x) = h_{L_1 + x}(x) = h_{L_1}(x) + x \cdot x,$$
so that $x \cdot x = 0$ and $L_2 = L_1$ once again.

We have shown that every convergent subsequence of ${\rm s}_\alpha K$ converges to $L_1$.
If the full sequence~${\rm s}_\alpha K$ does not converge, there is a subsequence $\gamma$ 
of~${\rm s}_\alpha K$ that stays some distance $\varepsilon > 0$ from $L_1$.  Since the
sequence~${\rm s}_\alpha K$ is uniformly bounded, so is the subsequence $\gamma$.  The Blaschke
selection theorem \cite[p. 97]{Webster} implies that $\gamma$, and therefore ${\rm s}_\alpha K$,
has a convergent subsequence $\gamma'$.  By the previous argument $\gamma'$ has limit $L_1$,
contradicting the construction of $\gamma$.
It follows that the original sequence ${\rm s}_\alpha K$ 
converges, and therefore must converge to the limit $L_1$.
\end{proof}

The condition that ${\rm s}_\alpha L = L$ for every subsequential limit $L$ is required 
for the proof of Theorem~\ref{fix} and does not hold for Steiner processes in general.  
Indeed, even when a Steiner process {\em converges}, 
it may not be the case that the limit is invariant under ${\rm s}_\alpha$.
In other words, the converse of Theorem~\ref{fix} is false.

A simple counterexample to the converse is constructed as follows.  Let $u$ and $v$
be distinct non-orthogonal unit vectors in $\RR^2$, and 
let $\alpha$ denote the sequence $\{u, v, v, \ldots \}$,
where $v$ is repeated forever.  If $K$ is any compact convex set in $\RR^2$,
then
${\rm s}_\alpha K = {\rm s}_v {\rm s}_u K$, since ${\rm s}_v$ is idempotent.
But ${\rm s}_v {\rm s}_u K \neq {\rm s}_v {\rm s}_u {\rm s}_v {\rm s}_u K$ in general
(for example, if $K$ is any line segment of positive length), so that
${\rm s}_\alpha {\rm s}_\alpha K \neq {\rm s}_\alpha K$.

\section{Steiner processes using a finite set of directions}
\label{finsetdir}

Suppose that $\alpha = \{u_1, u_2, \ldots \}$ is a sequence of unit vectors such that each $u_i$ is chosen
from a given {\bf finite} list of permitted directions $\{v_1, \ldots, v_m\}$.
\begin{theorem}  Let $K \in \KK_n$.  The sequence
${\rm s}_\alpha K$ has a limit $L \in \KK_n$.  Moreover, $L$ is symmetric under reflection
in each of the directions $v_i$ occurring infinitely often in the sequence.
\label{findir}
\end{theorem}
In other words, a Steiner process using a finite set of directions {\em always} converges.

\begin{proof} To begin, suppose that $K$ has nonempty interior. 
Without loss of generality (passing to a suitable tail of the sequence), we may
assume that each of the directions $v_i$ occurs infinitely often.
In view of Theorem~\ref{fix} it is then sufficient to show that every subsequential limit
of ${\rm s}_\alpha K$ is invariant under ${\rm s}_{v_i}$ for each $i$.  

Let $L$ denote the limit of some convergent subsequence of ${\rm s}_{\alpha} K$.  
Since the list of distinct vectors $v_i$ is finite, some $v_i$ occurs infinitely
often as the final iterate in this subsequence.  Without loss of generality,
relabel the directions $\{v_i\}$ so that $v_1$ is this recurring final direction.
Passing to the sub-subsequence $\{K_{i_j}\}$ 
where this
occurs, we are left with a sequence of the form
$$\{ K_{i_j} \}  = \{ {\rm s}_{v_1} {\rm s}_{u_{i_j - 1}} \cdots  {\rm s}_{u_1} K \}$$ 
where each $u_{i_j} = v_1$.  

Since every $K_{i_j}$ is an ${\rm s}_{v_1}$ symmetral, it is immediate that $L = \lim_j K_{i_j}$
is symmetric under reflection across $v_1^\perp$.

Note that each successor to $K_{i_j}$ in the original sequence $K_i$ has the form
$$K_{i_j +1}  = {\rm s}_{u_{i_j + 1}} {\rm s}_{v_1} {\rm s}_{u_{i_j - 1}} \cdots  {\rm s}_{u_1} K.$$
The direction $u_{i_j + 1}$ must attain one of the values $v_i$ infinitely often.  
Since ${\rm s}_{v_1} {\rm s}_{v_1} = {\rm s}_{v_1}$, we may
(without loss of generality) 
suppose this new direction is $v_2$, and that $v_2 \neq v_1$.  Let us pass further to the 
sub-subsequence
where every $u_{i_j + 1} = v_2$. 
It now follows that
$${\rm s}_{v_2} L = \lim_j {\rm s}_{v_2} K_{i_j} = \lim_j K_{i_j+1}.$$

Suppose that ${\rm s}_{v_2} L \neq L$.
In this case the strict monotonicity of $\Omega$ yields
$$\Omega({\rm s}_{v_2} L) - \Omega(L)> \varepsilon > 0$$
for some $\varepsilon > 0$.  By the continuity of $\Omega$ 
and the definition of $L$ there is an integer $M > 0$
such that
$$\Omega({\rm s}_{v_2} K_{i_j}) - \Omega(K_{i_t})> \frac{\varepsilon}{2} > 0$$
for all $j, t > M$.  But the monotonicity of $\Omega$ implies that
$$\Omega(K_{i_t}) \geq \Omega(K_{i_j+1}) = \Omega({\rm s}_{v_2} K_{i_j})$$
when $i_t > i_j$, a contradiction.  It follows that
$${\rm s}_{v_2} L = L.$$

More generally, suppose that $L = {\rm s}_{v_1}L = \cdots = {\rm s}_{v_k} L$, where $L$ is the limit
of the subsequence $K_{i_j}$.  For each $j$, let $Q_j$ be the first successor of $K_{i_j}$ in the
original sequence $K_i$ whose final iterated Steiner symmetrization uses a direction $v_t$ for $t > k$.
Again some particular $v_t$ must appear infinitely often as the final direction for the symmetrals $Q_j$.  
Without loss of generality, and passing to subsequences as needed, suppose this direction is always $v_{k+1}$.
Let $\tilde{Q}_j$ denote the immediate predecessor of each $Q_j$ in the original sequence $K_i$, 
so that $Q_j = s_{v_{k+1}} \tilde{Q}_j$.  

Again, passing to subsequences as needed, we may assume (by omitting repetitions) 
that each $Q_j$ corresponds to a distinct entry
of the original sequence $K_i$, so that $Q_t$ appears strictly 
later than $Q_j$ in the original sequence whenever $t > j$.

Since the subsequence $K_{i_j} \rightarrow L$ and $L$ has nonempty interior, Lemma~\ref{pmK} implies that,
for any given $\tau \in (0,1)$,
$$(1-\tau) L \subseteq K_{i_j} \subseteq (1+\tau) L$$
for sufficiently large $i_j$.  Since each $\tilde{Q}_j$ is a 
finite iteration of Steiner symmetrals of $K_{i_j}$ 
using only directions from the list $\{v_1, \ldots, v_k\}$, and 
because $L = {\rm s}_{v_1}L = \cdots = {\rm s}_{v_k} L$, it follows from the monotonicity of Steiner
symmetrization that 
$$(1-\tau) L \subseteq \tilde{Q}_j \subseteq (1+\tau) L$$
sufficiently large $j$, so that $\tilde{Q}_j \rightarrow L$ as well. 
It then follows from the monotonicity of ${\rm s}_{v_{k+1}}$ that
$$(1-\tau) {\rm s}_{v_{k+1}} L \subseteq Q_j \subseteq (1+\tau) {\rm s}_{v_{k+1}} L.$$
In other words, $Q_j \rightarrow {\rm s}_{v_{k+1}} L$.

Suppose that ${\rm s}_{v_{k+1}} L \neq L$.
In this case the strict monotonicity of $\Omega$ yields
$$\Omega({\rm s}_{v_{k+1}} L) - \Omega(L)> \varepsilon > 0$$
for some $\varepsilon > 0$.  
Since $Q_j \rightarrow {\rm s}_{v_{k+1}} L$ and $\tilde{Q}_j \rightarrow L$,
the continuity of $\Omega$ implies that
$$\Omega(Q_j) - \Omega(\tilde{Q}_t)> \frac{\varepsilon}{2} > 0$$
for all $j, t > M$, provided $M$ is sufficiently large.  But the monotonicity of $\Omega$ over the original sequence $K_i$ implies that
$$\Omega(\tilde{Q}_t) \geq \Omega(Q_j) = \Omega({\rm s}_{v_{k+1}} \tilde{Q}_j)$$
when $t > j$, a contradiction.  It follows that
$${\rm s}_{v_{k+1}} L = L.$$


It now follows that $L$ is symmetric under reflection in each of the directions $v_i$, so that 
${\rm s}_{\alpha} L = L$.  In other words $L$ is a fixed point for the process ${\rm s}_{\alpha}$.
Since this argument applies to every subsequential limit $L$ of ${\rm s}_\alpha K$, 
it follows from Theorem~\ref{fix} that these subsequential limits are identical, and 
that the original sequence $K_i$ converges to $L$.

Finally, suppose that $K$ has empty interior.  For each integer $m>0$, the parallel body 
$C_m = K + \frac{1}{m}B$ has interior, so the limit of ${\rm s}_\alpha C_m$ exists, by the previous argument.
Since each $C_{m} \supseteq C_{m+1}$, and
$$K = \bigcap_m C_m,$$
it follows from Proposition~\ref{bigcapconv} that the limit of ${\rm s}_\alpha K$ exists, and is given by
$${\rm s}_\alpha K = \bigcap_m {\rm s}_\alpha C_m.$$
Since each ${\rm s}_\alpha C_m$ is symmetric under reflection in each of the directions $v_i$, the
limit ${\rm s}_\alpha K$ is also symmetric under each of those reflections.
\end{proof} 

Recall that if $K \in \KK_n$ and $u \in \SS^{n-1}$, then ${\rm s}_u {\rm s}_u K = {\rm s}_u K$.  This is a trivial
consequence of the fact that ${\rm s}_u K$ is symmetric under reflection across $u^\perp$, so that any 
subsequent iteration
of ${\rm s}_u$ makes no difference.  On the other hand, given two non-identical and non-orthogonal directions
$u$ and $v$, it may easily happen that
$${\rm s}_u {\rm s}_v K \neq {\rm s}_u {\rm s}_v {\rm s}_u {\rm s}_v K.$$
More generally, there is no reason to believe that a Steiner process ${\rm s}_\alpha$ (whether finite or infinite) 
is idempotent.  However, the previous theorem implies that certain families of
Steiner processes are indeed idempotent.
\begin{corollary}Let $v_1, \ldots, v_m$ be unit directions in $\RR^n$,
and let $\alpha$ be a sequence of directions, each of whose entries is taken from among the $v_i$,
and in which each of the $v_i$ occurs infinitely often.  

The map ${\rm s}_\alpha:  \KK_n \rightarrow \KK_n$ given by $K \mapsto {\rm s}_\alpha K$ 
is well-defined and idempotent.
\label{idem}
\end{corollary}

Note that {\em every} direction in $\alpha$ must repeat infinitely often in the sequence to guarantee idempotence.

\begin{proof} It is an immediate consequence of Theorem~\ref{findir} that the map 
$K \mapsto {\rm s}_\alpha K$ is well-defined.  Since each ${\rm s}_\alpha K$ is symmetric under reflection
across each subspace $v_i^\perp$, it follows that  ${\rm s}_{v_i} {\rm s}_\alpha K = {\rm s}_\alpha K$ for each $i$,
so that ${\rm s}_\alpha {\rm s}_\alpha K = {\rm s}_\alpha K$.
\end{proof}

It follows from Theorem~\ref{findir} that {\em periodic} Steiner processes always converge to bodies
that are symmetric under the subgroup of $O(n)$ generated by reflections through a given repeated set of directions
$\{v_1, \ldots, v_m\}$.  More precisely, we have the following.
\begin{corollary}  Let $v_1, \ldots, v_m$ be unit directions in $\RR^n$,
and let $\alpha$ be the periodic sequence of directions given by 
\begin{align}
\alpha = \{\underbrace{v_1, \ldots, v_m}, \underbrace{v_1, \ldots, v_m}, \cdots \}.
\label{perseq}
\end{align}
Then the limit of ${\rm s}_{\alpha} K$ exists for every $K \in \KK_n$, and this limit
is symmetric under reflection across each subspace $v_i^\perp$, so that the Steiner process ${\rm s}_\alpha$
is idempotent.
\label{per}
\end{corollary}

A basis for $\RR^n$ is said to be {\em irrational} if the angles between any two vectors in the basis 
are irrational multiples of $\pi$.  The set of reflections across the coordinate planes of an irrational
basis generate a dense subgroup of $O(n)$.  Consequently, if a compact convex set $K$ is symmetric under
reflections across all of the directions from an irrational basis, then $K$ must be symmetric under {\em all}
reflections through the origin, so that $K$ must be a Euclidean ball, centered at the origin.

Applying the previous results to an irrational basis of directions leads to the 
following generalization of a periodic construction described in \cite[p. 98]{Egg}. 
\begin{corollary}  
Let $v_1, \ldots, v_m$ be a set of unit directions in $\RR^n$ that contains an irrational basis for $\RR^n$.
Suppose that $\alpha = \{u_1, u_2, \ldots \}$ is a sequence of unit vectors such that each $u_i$ is chosen
from the list of permitted directions $\{v_1, \ldots, v_m\}$, and such that each element of the irrational
basis appears infinitely often in the sequence $\alpha$.  Then the 
limit of ${\rm s}_{\alpha} K$ exists and is a Euclidean ball for every $K \in \KK_n$.
\label{irrat}
\end{corollary}
In particular, if a periodic sequence of the form~(\ref{perseq}) contains an irrational basis for $\RR^n$,
then ${\rm s}_{\alpha} K$ is a Euclidean ball for every $K \in \KK_n$.
For a generalization of this special case to arbitrary compact sets, see also 
\cite{burch-fort}.

\section{Open questions}

\noindent
{\em 1. Rate of convergence}\\

While Theorem~\ref{findir} guarantees convergence of infinite Steiner processes 
using a finite set of distinct directions,
there remain questions about the rate of convergence for different distributions of the permitted
set of directions.  For example, given three normal vectors $u,v,w$ 
to the edges of an equilateral triangle in $\RR^2$
and various choices of $\alpha$ such as 
\begin{align*}
\alpha & = \{\underbrace{u,v,w},
\underbrace{u,v,w}, \ldots\}, \\
\alpha & = \{\underbrace{u,v,w},v,
\underbrace{u,v,w,u,v,w},v,
\underbrace{u,v,w,u,v,w,u,v,w},v, \ldots\},\\
\alpha & = \{ u,v,w,
\underbrace{u,v,u,v},w,
\underbrace{u,v,u,v,u,v},w, \ldots\},
\end{align*}
how does the rate of convergence of ${\rm s}_\alpha K$ vary?  If instead $\alpha$ is determined by a sequence of
random choices from the set $\{u,v,w\}$, how is the rate of convergence 
related to the probability distribution for the choices of directions?\\

\noindent
{\em 2. More general classes of sets}\\

For most theorems regarding Steiner processes on convex bodies it is natural to ask whether similar results hold 
when the initial convex body is replaced by a more general kind of set, such as an arbitrary compact set in $\RR^n$
(see, for example, \cite{burch,burch-fort,vansch1,vansch2,volcic-symm}).
While the proof of Theorem~\ref{findir} above makes use of certain constructions that rely on convexity (such as
mixed volumes, and the equality condition for the Brunn-Minkowski inequality), one can still ask whether 
Theorem~\ref{findir} can be generalized  to Steiner processes on 
arbitrary compact sets in $\RR^n$.   
In \cite{burch-fort} Burchard and Fortier show that this is the case when
the finite set of repeated directions contains an irrational basis (as in Corollary~\ref{irrat}).
What happens if instead the finite set of directions generates a finite subgroup of reflections?  \\

\noindent
{\em 3. Cases of non-convergence}\\

There also remain many questions about the cases in which Steiner processes fail to converge.
In~\cite{BKLYZ} a convex body $K$ and a sequence
of directions $u_i$ are described for which the sequence of Steiner symmetrals
\begin{align*}
K_i = {\rm s}_{u_i} \cdots {\rm s}_{u_1} K
\end{align*}
fails to converge
in the Hausdorff topology. (For more such examples, see also \cite{burch-fort}.) 
More recently \cite{Bianchi-shape} it has been shown that
such examples converge in {\em shape}: there is a corresponding sequence 
of isometries $\psi_i$ such that
the sequence $\{ \psi_i K_i \}$ converges. 
However, many related questions remain open.  How does this limiting shape
depend on the initial body $K$ and the sequence $\alpha$ of 
symmetrizing directions?
What happens if $K$ is permitted to be an arbitrary
(possibly non-convex) compact set?


\providecommand{\bysame}{\leavevmode\hbox to3em{\hrulefill}\thinspace}
\providecommand{\MR}{\relax\ifhmode\unskip\space\fi MR }
\providecommand{\MRhref}[2]{%
  \href{http://www.ams.org/mathscinet-getitem?mr=#1}{#2}
}
\providecommand{\href}[2]{#2}

\end{document}